\documentclass[preprint,12pt]{elsarticle}

\usepackage{amssymb}
\usepackage{color}
\usepackage{mathptmx}       % selects Times Roman as basic font
\usepackage{helvet}         % selects\textbf{\textbf{•}} Helvetica as sans-serif font
\usepackage{courier}        % selects Courier as typewriter font
\usepackage{type1cm}        % activate if the above 3 fonts are
\usepackage{makeidx}         % allows index generation
\usepackage{graphicx}        % standard LaTeX graphics tool
\usepackage{multicol}        % used for the two-column index
\usepackage[bottom]{footmisc}% places footnotes at page bottom
\usepackage{amsmath}
\usepackage{amsfonts}
\usepackage{amsthm}

\newtheorem{theorem}{Theorem}
\newtheorem{corollary}[theorem]{Corollary}
\newtheorem{lemma}[theorem]{Lemma}

\newtheorem{proposition}[theorem]{Proposition}

\numberwithin{equation}{section}

\journal{Stochastic Processes and their Applications}

\begin{document}

\begin{frontmatter}



\title{Conformal restriction: the radial case}


\author{Hao Wu}%\thanks{H.W.'s work is funded by the Fondation CFM JP Aguilar pour la recherche. The author acknowledges also the support and hospitality of FIM at ETH Z\"urich}}

\address{Universit\'e Paris-Sud}

\begin{abstract}
We describe all random sets that satisfy the radial conformal restriction property, therefore providing the analogue in the radial case of results of
Lawler, Schramm and Werner in the chordal case.

\end{abstract}


\end{frontmatter}

\newcommand{\eps}{\epsilon}
\newcommand{\ov}{\overline}
\newcommand{\U}{\mathbb{U}}
\newcommand{\T}{\mathbb{T}}
\newcommand{\HH}{\mathbb{H}}
\newcommand{\LA}{\mathcal{A}}
\newcommand{\LC}{\mathcal{C}}
\newcommand{\LF}{\mathcal{F}}
\newcommand{\LK}{\mathcal{K}}
\newcommand{\LE}{\mathcal{E}}
\newcommand{\LL}{\mathcal{L}}
\newcommand{\R}{\mathbb{R}}
\newcommand{\C}{\mathbb{C}}
\newcommand{\N}{\mathbb{N}}
\newcommand{\Z}{\mathbb{Z}}
\newcommand{\E}{\mathbb{E}}
\newcommand{\PP}{\mathbb{P}}
\newcommand{\QQ}{\mathbb{Q}}
\newcommand{\MR}{MR}
\newcommand{\giv}{\,|\,}

% \linenumbers

\section{Introduction}

The present paper is a write-up of the ``radial'' counterpart of some of the results derived in the ``chordal'' setting in the paper
 \cite{LawlerSchrammWernerConformalRestriction} by Lawler, Schramm and Werner. The goal is to describe the laws of all random sets that satisfy
 a certain radial conformal restriction property.

Let us describe without further ado this property, and the main result of the present paper:
Consider the unit disc $\U$ and we fix a boundary point $1$ and an interior point the origin. We will study closed random subsets $K$ of $\overline{\U}$ such that:
\begin{itemize}
\item $K$ is connected, $\C\setminus K$ is connected, $K\cap\partial\U=\{1\}$, $0\in K$ .
\item For any closed subset $A$ of $\overline{\U}$ such that $A=\overline{\U\cap A}$, $\U\setminus A$ is simply connected, contains the origin and has 1 on the boundary, the law of $\Phi_A(K)$ conditioned on $(K\cap A=\emptyset)$ is equal to to law of $K$ where $\Phi_A$ is the conformal map from $\U\setminus A$ onto $\U$ that preserves 1 and the origin (see Figure \ref{fig::radial_conf_restriction}).
\end{itemize}
The law of such a set $K$ is called a radial restriction measure, by analogy with the chordal restriction measures defined in  \cite{LawlerSchrammWernerConformalRestriction}.

\begin{figure}[ht!]
\begin{center}
\includegraphics[width=0.7\textwidth]{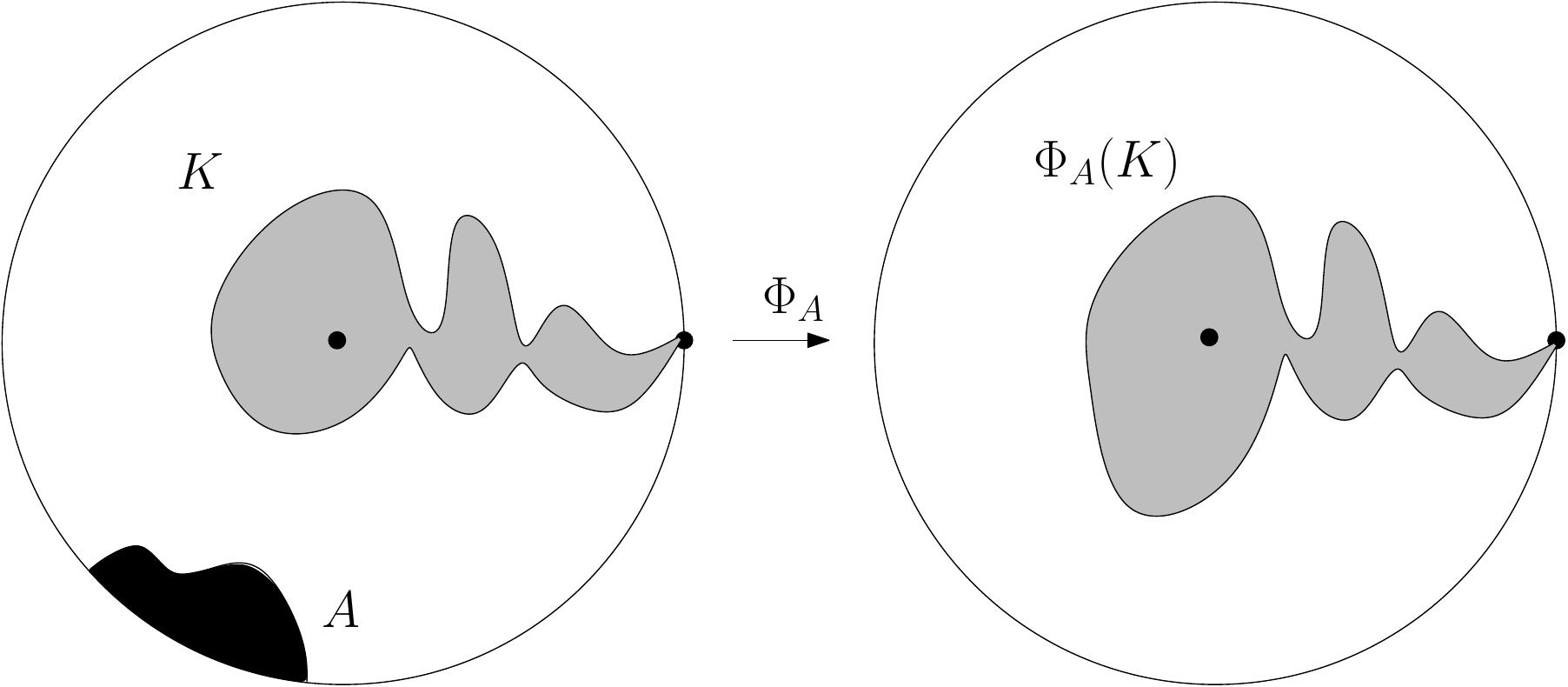}
\end{center}
\caption{\label{fig::radial_conf_restriction}$\Phi_A$ is the conformal map from $\U\setminus A$ onto $\U$ that preserves 0 and 1. Conditioned on $(K\cap A=\emptyset)$, $\Phi_A(K)$ has the same law as $K$.}
\end{figure}

The main result of the present paper is the following classification and description of all radial restriction measures.
\begin{theorem}\label{thm::radial_restriction}
\begin {enumerate}
\item {(Characterization).}
A radial restriction measure is fully characterized by a pair of real numbers $(\alpha,\beta)$ such that
\begin{equation*}\PP\bigl[K\cap A=\emptyset\bigr]=|\Phi_A'(0)|^{\alpha}\Phi_A'(1)^{\beta}\end{equation*}
where $A$ is any closed subset of $\overline{\U}$ such that $A=\overline{\U\cap A}$, $\U\setminus A$ is simply connected, contains the origin and has 1 on the boundary, and $\Phi_A$ is the conformal map from $\U\setminus A$ onto $\U$ that preserves $0$ and $1$. We denote the corresponding radial restriction measure by $\PP\left({\alpha,\beta}\right)$.
\item {(Existence).}
The measure $\PP\left({\alpha,\beta}\right)$ exists if and only if
\[\beta\ge \frac{5}{8},\quad \alpha\le \xi(\beta)=\frac{1}{48}\left((\sqrt{24\beta+1}-1)^2-4\right).\]
\end {enumerate}
\end{theorem}
We shall give an explicit construction of the measures  $\PP\left({\alpha,\beta}\right)$ for all these admissible values of $\alpha$ and $\beta$. The function
$\xi (\beta)$ is (as could be expected) the so-called disconnection exponent associated with the half-plane exponent $\beta$ (see \cite{LawlerWernerUniversalityExponent,LawlerSchrammWernerExponent1, LawlerSchrammWernerExponent2,LawlerSchrammWernerExponent3}).

It is worth observing that $|\Phi_A'(0)| \ge 1$ and that $\Phi'_A (1)\le 1$. In Theorem \ref {thm::radial_restriction}, we see that the value of $\beta$ is necessarily positive (and that therefore $\Phi'_A (1)^\beta \le 1$), but the value of $\alpha$ can be negative or positive (as long as $\alpha \le \xi (\beta)$), so that
$|\Phi'_A(0)|^\alpha$ can be greater than one (but of course, the product $|\Phi'_A (0)|^\alpha \Phi'_A (1)^\beta $ cannot be greater than one which is guaranteed by the condition $\alpha\le\xi(\beta)$).

This theorem is the counterpart of the classification of chordal restriction measures in  \cite{LawlerSchrammWernerConformalRestriction} that we shall recall in the next section.
It is worth noticing already that while the class of chordal conformal restriction measures was parametrized by a single parameter $\beta \ge 5/8$, the class of radial restriction samples is somewhat larger as it involves the additional parameter $\alpha$. This can be rather easily explained by the fact that the radial restriction property is in a sense weaker than the chordal one. It involves an invariance property of the probability distribution under the action of the semi-group of conformal transformations that preserve both an inner point and a boundary point of the disc. In the chordal case, the semi-group of transformations were those maps that preserve two given boundary points (which is a larger family). Another way to see this is that the chordal restriction samples in the upper half-plane are scale-invariant, while the radial ones aren't. However, and this will be apparent in the latter part of the proof of Theorem \ref{thm::radial_restriction}, chordal restriction samples of parameter $\beta$ can be
viewed as limits of radial ones with parameters $(\alpha, \beta)$ (for all admissible $\alpha$'s), in the same way as chordal SLE can be viewed as the limit of radial SLE when the inner point converges to the boundary of the domain.

\medbreak
These results have been discussed and mentioned before, at least partially, in e-mail exchanges, lectures and discussions by a number of mathematicians, including of course Lawler, Schramm and Werner, and also Dub\'edat or Gruzberg.
In fact, reference 31. in the paper \cite{LawlerSchrammWernerConformalRestriction} written in 2003 by Lawler, Schramm and Werner is precisely a paper ``in preparation'' with the very same title as the present one. I wish to hereby thank Greg Lawler and Wendelin Werner for letting me write up the present paper and work out the details of the proofs.

\section{Preliminaries}
We now briefly recall some background material that will be needed in our proofs, concerning chordal or radial SLE and their SLE$_\kappa (\rho)$ variants, Brownian loop-soups as well as chordal restriction measures.
When $K$ is a subset of $\C$ and $x\in\C$, we denote $x+K$ as the set $\{x+z :  z\in K\}$ and $xK$ as the set $\{xz : z\in K\}$.
\subsection{Chordal Loewner chains and SLE}
Suppose $(W_t,t\ge 0)$ is a real-valued continuous function. For each $z\in\overline{\HH}$, define $g_t(z)$ as the solution to the chordal Loewner ODE:
\[\partial_t g_t(z)=\frac{2}{g_t(z)-W_t},\quad g_0(z)=z.\]
Write $\tau(z)=\sup\{t\ge 0: \inf_{s\in[0,t]}|g_s(z)-W_s|>0\}$ and $K_t=\{z\in\HH: \tau(z)\le t\}$. Then $g_t$ is the unique conformal map from $\HH\setminus K_t$ onto $\HH$ such that $|g_t(z)-z|\to 0$ as $z\to\infty$. And $(g_t,t\ge 0)$ is called the chordal Loewner chain generated by the driving function $(W_t,t\ge 0)$. In fact, we have $(g_t(z)-z)z\to 2t \text{ as } z\to\infty$.
\medbreak
SLE curves are introduced by Oded Schramm as candidates of scaling limits of discrete statistical physics models (see \cite{SchrammFirstSLE}). A chordal SLE$_{\kappa}$ is defined by the random family of chordal conformal maps $g_t$ when $W=\sqrt{\kappa}B$ where $B$ is a standard one-dimensional Brownian motion. It is proved that there exists a.s. a continuous curve $\eta$ such that for each $t\ge 0,$ $\HH\setminus K_t$ is the unbounded connected component of $\HH\setminus \eta([0,t])$ (see \cite{RohdeSchrammSLEBasicProperty}).

\medbreak

Chordal SLE$_{\kappa}(\rho)$ processes are variants of SLE$_{\kappa}$ process. For simplicity, we will here only describe the SLE$_\kappa (\rho)$ processes with just one additional
force point:  It is the measure on the random family of conformal maps $g_t$ generated by chordal Loewner chain with $W_t$ replaced by the solution to the system of SDEs:
\begin{equation*}
\begin{split}
dW_t&=\sqrt{\kappa}dB_t+\frac{\rho}{W_t-V_t}dt;\\
dV_t&=\frac{2}{V_t-W_t}dt, \quad V_0=x \not= 0, \quad (W_t - V_t) / (W_0 - V_0) \ge 0.
\end{split}
\end{equation*}

When $\kappa>0, \rho>-2$, there is a pathwise unique solution to the above SDEs. The force point is repelling when $\rho$ is positive while it is attracting when $\rho$ is negative.
There exists a.s. a continuous curve $\eta$ in $\overline{\HH}$ from 0 to $\infty$ associated to the SLE$_{\kappa}(\rho)$ process (see \cite{MillerSheffieldIG1}).

In the limit when $x \to 0+$ (respectively $0-$), the process has a limit that is scale-invariant in distribution. This enables to define the corresponding SLE$_\kappa (\rho)$ (referred to as SLE$_\kappa^R (\rho)$ or SLE$_\kappa^L (\rho)$ to indicate if the force-point is to the right or to the left of the driving point) from a boundary point of a simply connected domain to another by conformal invariance, just as for ordinary SLE$_\kappa$.

\subsection{Chordal restriction samples}\label{sec::pre_chordal_restriction}
We now recall briefly some facts from \cite{LawlerSchrammWernerConformalRestriction}.
Consider the upper half plane $\HH$ and we fix two boundary points $0$ and $\infty$. A (two-sided) chordal restriction sample is a closed random subset of $\overline{\HH}$ such that
\begin{itemize}
\item $K$ is connected, $\C\setminus K$ is simply connected, $K\cap \R=\{0\}$, and $K$ is unbounded.
\item For any closed subset $A$ of $\overline{\HH}$ such that $A=\overline{\HH\cap A}$, $\HH\setminus A$ is simply connected, $A$ is bounded and $0\not\in A$, the law of $\Psi_A(K)$ conditioned on $(K\cap A=\emptyset)$ is equal to the law of $K$ where $\Psi_A$ is any given conformal map from $\HH\setminus A$ onto $\HH$ that preserves 0 and $\infty$.
\end{itemize}
Note that this second property in the case where $A = \emptyset$ shows that the law of $K$ is scale-invariant (ie. that $K$ and $\lambda K$ have the same distribution for any fixed positive $\lambda$).
It is proved that the chordal restriction measures form a one-parameter family $(\QQ_{\beta})$, such that for all $A$ as before,
\[\QQ_\beta\bigl[K\cap A=\emptyset\bigr]=\Psi_A'(0)^{\beta}\]
where $\Psi_A$ is the conformal map from $\HH\setminus A$ onto $\HH$ that preserves 0 and $\Psi_A(z)/z\to 1$ as $z\to\infty$ (see \cite{LawlerSchrammWernerConformalRestriction}). In that paper, it is proved that the chordal conformal restriction measure $\QQ_{\beta}$ exists if and only if $\beta\ge 5/8$.

We would like to make the following remarks that will be relevant for the present paper:
\begin {enumerate}
\item
Chordal restriction samples can be defined in any simply connected domain $H \not= \C$ by conformal invariance (using the fact that their law in $\HH$ is scale-invariant: $K$ and $\lambda K$ have the same law for any fixed positive constant $\lambda$).
For instance, if  $H$ is such a simply connected domain and  $z,w$ are two different boundary points, the chordal restriction sample in $H$ connecting $z$ and $w$ is the image of chordal restriction sample in $\HH$ under any given conformal map $\phi$ from $\HH$ onto $H$ that sends the pair $(0,\infty)$ to $(z,w)$.
\item
In the proof of the construction of these (two-sided) chordal restriction samples, an important role is played by
the related ``right-sided chordal restriction samples'', that we shall also use at some point in the present paper.
These are a closed random subset $K$ of $\overline{\HH}$ such that
\begin{itemize}
\item $K$ is connected, $\C\setminus K$ is connected, $K\cap\R=(-\infty,0]$.
\item For any closed subset $A$ of $\overline{\HH}$ such that $A=\overline{\HH\cap A}$, $\HH\setminus A$ is simply connected, $A$ is bounded and $A\cap\R\subset (0,\infty)$, the law of $\Psi_A(K)$ conditioned on $(K\cap A=\emptyset)$ is equal to the law of $K$ where $\Psi_A$ is any conformal map from $\HH\setminus A$ onto $\HH$ that preserves 0 and $\infty$.
\end{itemize}
It is clear that the domain to the left of the right boundary of chordal restriction sample is a right-sided restriction sample. Precisely, suppose $K$ is the closure of the union of the domains between $\R_-$ and the right boundary of a (two-sided) chordal restriction sample, then $K$ is a right-sided restriction sample. In fact, there exists a one-parameter family $\QQ^+_{\beta}$ such that
\[\QQ^+_{\beta}\bigl[K\cap A=\emptyset\bigr]=\Psi_A'(0)^{\beta}\]
where $\Psi_A$ is the conformal map from $\HH\setminus A$ onto $\HH$ that preserves 0 and $\Psi_A(z)/z\to 1$ as $z\to\infty$. $\QQ^+_{\beta}$ exists if and only if $\beta\ge 0$. We usually ignore the trivial case $\beta=0$ where $K=\R_-$.

One example of right-sided restriction sample is given by SLE$_{8/3}^L (\rho)$:  Let $\eta$ be such a process in $\overline{\HH}$ from 0 to $\infty$. Let $K$ be the closure of the union of domains between $\eta$ and $\R_-$. Then $K$ is a right-sided restriction sample with exponent $\beta=(\rho+2)(3\rho+10)/32$. Conversely, let $K$ be a right-sided restriction sample with exponent $\beta>0$, then the right boundary of $K$ is an SLE$_{8/3}^L(\rho)$ process with
\begin{equation}\label{eqn::rho_beta}\rho=\rho(\beta)=\frac{2}{3}(\sqrt{24\beta+1}-1)-2.\end{equation}

\item We have just seen the the right boundary  of a two-sided restriction sample is an SLE$_{8/3}^L (\rho)$ process. It is also possible to describe the conditional law of the left boundary given the right boundary:
Denote $L_r$ as the domain between $\R_-$ and the right boundary of $K$. Then, given this right boundary, the conditional law of the left boundary of $K$ is an SLE$_{8/3}^R(\rho-2)$ from  $0$ to $\infty$ in $L_r$ (see \cite{WernerGirsanov}). In fact, we shall \textit{construct} our radial restriction samples using the radial analogue of this recipe.

\item
Let $C(K)$ be the cut point set of $K$ i.e. the set of points $x$ in $K$ such that  $K \setminus \{ x \}$ is not connected.
Note that $C(K)$ is the intersection of the right and left boundaries of $K$. It turns out that the right and left boundaries of $K$ can be coupled with a Gaussian Free Field as two flow lines, which enables to prove (see \cite[Theorem 1.5]{MillerWuSLEIntersection}) that the Hausdorff dimension of $C(K)$ is almost surely
equal to $(25-u^2)/12$ where $u=\sqrt{24\beta+1}-1$,
when $5/8\le \beta\le 35/24$, whereas $C(K)=\emptyset$ almost surely when $\beta>35/24$.

\item
It is possible to describe the half-plane Brownian non-intersection exponents $\tilde{\xi}$ in terms of restriction measures.
For instance, consider two independent chordal restriction samples $K_1$ and $K_2$ with exponent $\beta_1,\beta_2$ respectively. One can derive that, conditioned on $(K_1\cap K_2=\emptyset)$ (viewed as the limit of $K_1\cap(x+K_2)\cap B(0,R)=\emptyset$ as $x\to 0,R\to\infty$), the ``inside'' of $K_1\cup K_2$ has the same law as a chordal restriction sample of exponent $\tilde \xi (\beta_1, \beta_2)$.

\item
It is possible to use restriction samples in order to describe the law of SLE$_{\kappa}(\rho)$ processes as SLE$_{\kappa}$ processes conditioned not to intersect a chordal restriction sample. For details, see \cite[Equations (9),(10)]{WernerGirsanov}.
\end {enumerate}

\subsection{Brownian loop soup}
We now briefly recall some results from \cite{LawlerWernerBrownianLoopsoup}.
It is well known that Brownian motion in $\C$ is conformally invariant. Let us now define for all $t \ge 0$, the law $\mu_t (z,z)$ of the two-dimensional Brownian bridge of time-length
$t$ that starts and ends at $t$ and define
\[\mu^{loop}=\int_{\C} \int_0^\infty  dz \frac {dt}{t} \mu_t (z,z)\]
where $dz$ is the Lebesgue measure in $\C$ that we view as a measure on \textit{unrooted} loops modulo time-reparametrization  (see \cite{LawlerWernerBrownianLoopsoup}). Then, $\mu^{loop}$ inherits a striking conformal invariance property.
More precisely, if for any subset $D\subset\C$, one defines the Brownian loop measure $\mu^{loop}_D$ in $D$ as the restriction of $\mu^{loop}$ to the set of loops contained in $D$, then it is shown in \cite{LawlerWernerBrownianLoopsoup}:
\begin{itemize}
\item For two domains $D'\subset D$, $\mu^{loop}_{D}$ restricted to the loops contained in $D'$ is the same as $\mu^{loop}_{D'}$ (this is a trivial consequence of the definition of these measures).
\item For two simply connected domains $D_1,D_2$, let $\Phi$ be a conformal map from $D_1$ onto $D_2$, then the image of $\mu^{loop}_{D_1}$ under $\Phi$ has the same law as $\mu^{loop}_{D_2}$ (this non-trivial fact is inherited from the conformal invariance of planar Brownian motion).
\end{itemize}

From these two properties, if we denote $\mu_{\U}^0$ as $\mu^{loop}_{\U}$ restricted to the loops surrounding the origin, then it is further noted in \cite{WernerSelfavoidingLoop} that
\begin{equation}\label{eqn::brownian_loop_exit}
\mu^0_{\U}(\gamma\not\subset U)=\log \Phi'(0)
\end{equation}
where $U$ is any simply connected subset of $\U$ that contains the origin and $\Phi$ is the conformal map from $U$ onto $\U$ that preserves the origin and $\Phi'(0)>0$.

For $c>0$, let $(\gamma_j,j\in J)$ be a Poisson point process with intensity $c\mu_{\U}^0$, then, from Equation \eqref{eqn::brownian_loop_exit}, we have that
\[\PP\bigl[\gamma_j\subset U,\forall j\in J\bigr]=\exp\left(-c\mu_{\U}^0(\gamma\not\subset U)\right)=\Phi'(0)^{-c}\]
where $U$ is any simply connected subset of $\U$ that contains the origin and $\Phi$ is the conformal map from $U$ onto $\U$ that preserves the origin and $\Phi'(0)>0$.

\subsection{Radial Loewner chains and SLE}
Suppose $(W_t,t\ge 0)$ is a real-valued continuous function. For each $z\in\overline{\U}$, define $g_t(z)$ as the solution to the radial Loewner ODE:
\[\partial_t g_t(z)=g_t(z)\frac{e^{iW_t}+g_t(z)}{e^{iW_t}-g_t(z)},\quad g_0(z)=z.\]
Write $\tau(z)=\sup\{t\ge 0: \inf_{s\in[0,t]}|g_s(z)-e^{iW_s}|>0\}$ and $K_t=\{z\in\overline{\U}: \tau(z)\le t\}$. Then $g_t$ is the unique conformal map from $\U\setminus K_t$ onto $\U$ such that $g_t(0)=0,g_t'(0)>0$. And $(g_t,t\ge 0)$ is called the radial Loewner chain generated by the driving function $(W_t,t\ge 0)$. In fact, we have $g'_t(0)=e^t$.
\medbreak
Before introducing the radial SLE, let us first define some special Loewner chains that will be of use later on.
We want to define a radial Loewner curve $\eta$ such that, for any $t>0$, the future part of the curve $\eta([t,\infty))$ under $g_t$ is exactly $\eta$ up to a rotation of the disc.
Precisely, fix $\theta\in (0,2\pi)$, define the driving function $W^{\theta}_t=\theta-t\cot\frac{\theta}{2}$. Let $(g_t,t\ge 0)$ be the radial Loewner chain generated by $W^{\theta}$. And define $f_t(\cdot)=g_t(\cdot)/g_t(1)$. Then there exists a continuous curve $\eta^{\theta}$ started from $e^{i\theta}$ and ended at the origin such that $g_t$ is the conformal map from $\U\setminus \eta^{\theta}([0,t])$ and $g_t(0)=0,g_t'(0)=e^t$. From the radial Loewner ODE, we have that $g_t(1)=e^{i(W_t-\theta)}$, and $f_t(\eta^{\theta}(t))=e^{i\theta}$. Further, for any $t,s>0$, $f_t(\eta^{\theta}([t,t+s]))=\eta^{\theta}([0,s])$. We call $\eta^{\theta}$ as \textit{perfect radial curve} started from $e^{i\theta}$. Note that
\begin{equation}\label{eqn::perfect_two_derivatives}
|f'_t(0)|=e^t,\quad f'_t(1)=\exp(-\frac{t}{1-\cos\theta}).
\end{equation}

\medbreak
A radial SLE$_{\kappa}$ is defined by the random family of radial conformal maps $g_t$ when $W=\sqrt{\kappa}B$ where $B$ is a standard one-dimensional Brownian motion. It is proved that there exists a.s. a continuous curve $\eta$ such that for each $t\ge 0,$ $\U\setminus K_t$ is the connected component of $\U\setminus \eta([0,t])$ containing the origin (this is due to the absolute continuity relation between radial and chordal SLEs and the corresponding results for chordal SLEs).

Let us briefly focus on radial SLE$_{8/3}$. Let $\eta$ be an SLE$_{8/3}$ in $\U$ from 1 to the origin. It is known (see \cite[Section 6.5]{LawlerConformallyInvariantProcesses}) that
\begin{equation}\label{eqn::radial_sle_8/3}
\PP\bigl[\eta\cap A=\emptyset\bigr]=|\Phi_A'(0)|^{{5}/{48}}\Phi_A'(1)^{{5}/{8}}
\end{equation}
where $A$ is any closed subset of $\overline{\U}$ such that $A=\overline{\U\cap A}$, $\U\setminus A$ is simply connected, contains the origin and has 1 on the boundary; $\Phi_A$ is the conformal map from $\U\setminus A$ onto $\U$ that preserves the origin and the boundary point $1$. This result follows from a standard martingale computation for radial SLE$_{8/3}$. This will ensure that the measure that we will call $\PP (5/48, 5/8)$ does exist.

\medbreak
We will also make use of a radial version of SLE$_\kappa (\rho)$ processes. For simplicity, let us just define the radial SLE$_{\kappa}(\rho)$ process with only one force point.
It is the measure on the random family of conformal maps $g_t$ generated by radial Loewner chain with $W_t$ replaced by the solution to the system of SDEs:
\begin{equation}\label{eqn::radial_loewner_sde}
\begin{split}
dW_t&=\sqrt{\kappa}dB_t+\frac{\rho}{2}\cot(\frac{W_t-V_t}{2}) dt;\\
dV_t&=-\cot(\frac{W_t-V_t}{2}) dt, \quad V_0=x\in (0,2\pi).
\end{split}
\end{equation}
When $\kappa>0, \rho>-2$, there is a pathwise unique solution to the above SDEs. And there exists a.s. a continuous curve $\eta$ in $\overline{\U}$ from 1 to 0 associated to the radial SLE$_{\kappa}(\rho)$ process \cite{SchrammWilsonSLECoordinatechanges,ZhanReversalRadialSLE, MillerSheffieldIG4}. Note that, in the radial case, a right force point $e^{ix}$ with $x\in (0,2\pi)$ can also be viewed as a left force point $e^{i(2\pi-x)}$. Thus, in constrast with the chordal case, we do not use the terminology of ``left" and ``right" force point for the radial case. Let $x\to 0+$ (resp. $x\to 2\pi-$), the process has a limit and we call this limit process as radial SLE$_{\kappa}(\rho)$ in $\overline{\U}$ from 1 to 0 with force point $1^+$ (resp. $1^-$). It is worthwhile to point out that the perfect curve started from $e^{i\theta}$ can also be viewed as radial SLE$_0(-2)$ process with $W_0=\theta, V_0=0$.

\section{Characterization}
The present section will be devoted to the proof of the characterization part of our main theorem.

Let $\LA^r$ be the set of all closed $A\subset \overline{\U}$ such that $A=\overline{A\cap\U}$, $\U\setminus A$ is simply connected, contains the origin and has 1 on the boundary. For any $A\in\LA^r$, define $\Phi_A$ as the conformal map from $\U\setminus A$ onto $\U$ such that preserves 1 and the origin. We usually call $\log|\Phi_A'(0)|$ as the capacity of $A$ in $\U$ seen from the origin. Generally, for any domain $U\subset\C$, a closed subset $A\subset\overline{U}$, and a point $z\in\U\setminus A$, the capacity of $A$ in $U$ seen from $z$ is $\log \Phi'(z)$ where $\Phi$ is the conformal map from the connected component of $U\setminus A$ that contains $z$ onto $\U$ and is normalized at $z:$ $\Phi(z)=0,\Phi'(z)>0$.

Let $\Omega$ be the collection of closed subsets $K$ of $\overline{\U}$ such that $K$ is connected, $\C\setminus K$ is connected and $1\in K$, $0\in K$. Endow $\Omega$ with the $\sigma$-field generated by the family of events of the type  $\{ K\in\Omega \  : \  K\cap A=\emptyset\}$ where $A\in\LA^r$ (note that this $\sigma$-field coincides with the $\sigma$-field generated by Hausdorff metric on $\Omega$, this is similar to the chordal case). It is clear that this family of events is closed under finite intersection, so that, just as in the chordal case, we know that:
\begin{lemma}\label{lem::unicity_law}
If $\PP$ and $\PP'$ are two probability measures on $\Omega$ such that $\PP\bigl[K\cap A=\emptyset\bigr]=\PP'\bigl[K\cap A=\emptyset\bigr]$ for all $A\in\LA^r$, then $\PP=\PP'$.
\end{lemma}

Note that we endow $\LA^r$ with Hausdorff metric, and recall that $K\cap\partial\U=\{1\}$, thus function $A\mapsto \PP[K\cap A=\emptyset]$ is continuous on $\LA^r$. We will implicitly use this fact later in the paper.
%%%%%%%%%%%%%%%%%%%%%%%%%%%%%%%%%%%%%%%%%%%%%%%%%%%%%%%%%%%%%%%%%%%%%%%%%%%%%%%%%%%%%%%%%%%%%%%%%%%%%%%%
%%\begin{lemma}If radial restriction measure $\PP(\alpha_0,\beta_0)$ exists for some $\alpha_0,\beta_0\in\R$, define $\alpha(\beta_0)=\sup\{\alpha: \PP(\alpha,\beta_0)\text{ exists }\}$. Then $\PP\left(\alpha(\beta_0),\beta_0\right)$ exists.
%%\end{lemma}
%%\begin{proof}
%%Fix $\beta_0$. Let $K(\alpha)$ be a radial restriction set sampled according to $\PP(\alpha,\beta_0)$ for $\alpha<\alpha(\beta_0)$. From construction in Lemma \ref{lem::brownian_loop_soup}, we know that for any $\alpha_1<\alpha_2<\alpha(\beta_0)$, we can couple $K(\alpha_1)$ and $K(\alpha_2)$ in the way that $K(\alpha_2)\subset K(\alpha_1)$.
%%
%%In fact, we can couple all $K(\alpha)$ for $\alpha<\alpha(\beta_0)$ such that, for any $\alpha_1<\alpha_2$, $K(\alpha_2)\subset K(\alpha_1)$. Under this coupling, we set $K=\cap_{\alpha<\alpha(\beta_0)}K(\alpha)$. Then it is clear that $K$ has the same law as $\PP\left(\alpha(\beta_0),\beta_0\right)$.
%%\end{proof}
%%%%%%%%%%%%%%%%%%%%%%%%%%%%%%%%%%%%%%%%%%%%%%%%%%%%%%%%%%%%%%%%%%%%%%%%%%%%%%%%%%%%%%%%%%%%%%%%%%%%%%%%%%%%%

It will be useful to use our perfect radial curves. The following fact is the analogue of the fact derived through \cite[Equation (3.1)]{LawlerSchrammWernerConformalRestriction}:
\begin{lemma}\label{lem::definition_nu}
Fix $\theta\in(0,2\pi)$ and let $\eta^{\theta}$ be the perfect radial curve started from $e^{i\theta}$. Let $K$ be a radial restriction sample, then there exists $\nu(\theta)\in(0,\infty)$ such that, for all $t\ge 0$,
\[\PP\bigl[K\cap \eta^{\theta}([0,t])=\emptyset\bigr]=\exp(-\nu(\theta)t).\]
\end{lemma}

\begin{figure}[ht!]
\begin{center}
\includegraphics[width=0.7\textwidth]{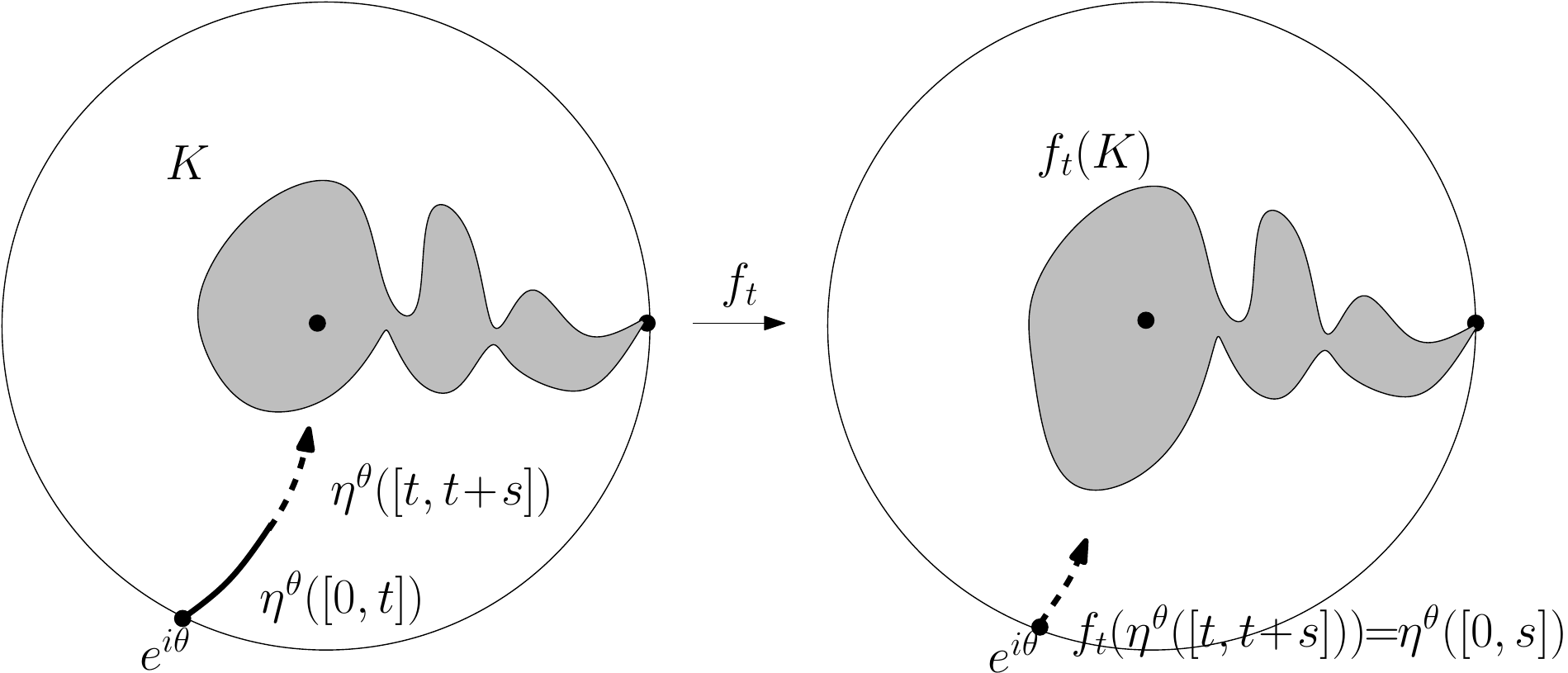}
\end{center}
\caption{\label{fig::perfect_radial_curve} Conditioned on $(K\cap \eta^{\theta}([0,t])=\emptyset)$, $f_t(K)$ has the same law as $K$.}
\end{figure}

\begin{proof} (See Figure \ref{fig::perfect_radial_curve})
Recall that $f_t$ is the conformal map from $\U\setminus \eta^{\theta}([0,t])$ onto $\U$ such that $f_t(0)=0,|f'_t(0)|=e^t, f_t(\eta^{\theta}(t))=e^{i\theta}$ and we also have that $f_t(\eta^{\theta}([t,t+s]))=\eta^{\theta}([0,s])$ for any $t,s>0$. Then, for any $t,s>0$, by the property of radial restriction sample, we have that
\begin{eqnarray*}
\lefteqn{\PP\bigl[K\cap \eta^{\theta}([0,t+s])=\emptyset\giv K\cap \eta^{\theta}([0,t])=\emptyset\bigr]}\\
&=&\PP\bigl[K\cap f_t(\eta^{\theta}([t,t+s]))=\emptyset\bigr]
=\PP\bigl[K\cap \eta^{\theta}([0,s])=\emptyset\bigr].
\end{eqnarray*}
Thus, for any $t,s>0$, we have
\[\PP\bigl[K\cap \eta^{\theta}([0,t+s])=\emptyset\bigr]=\PP\bigl[K\cap \eta^{\theta}([0,t])=\emptyset\bigr]\times \PP\bigl[K\cap \eta^{\theta}([0,s])=\emptyset\bigr].\]
Together with the fact that the function $t\mapsto \PP[K\cap \eta^{\theta}([0,t])=\emptyset]$ is continuous, we have that
\[\PP\bigl[K\cap \eta^{\theta}([0,t])\bigr]=\exp(-\nu(\theta)t)\] for some $\nu(\theta)\in[0,\infty]$. If $\nu(\theta)=\infty$, then $K\cap\eta^{\theta}([0,t])\neq\emptyset$ a.s., for all $t>0$. However $\cap_{t>0}\eta^{\theta}([0,t])=\{e^{i\theta}\}$ and $e^{i\theta}\not\in K$. This rules out the possibility of $\nu(\theta)=\infty$. If $\nu(\theta)=0$, then $K\cap \eta^{\theta}([0,\infty])=\emptyset$ a.s.. This is also impossible since $0\in K$ and $\eta^{\theta}$ ends at the origin.
\end{proof}

We would like to note at this point that in the chordal case, the analogous quantity was obviously constant because of scale-invariance of the chordal restriction measures in the upper half-plane. In the present radial case, this is not going to be the case. In particular, care will be needed to show that $\theta \mapsto \nu (\theta)$ is continuously differentiable.

\medbreak

%\section{Characterization of radial restriction measure}
%In this section, w
We are now ready to prove the first part of Theorem \ref{thm::radial_restriction} that we now state as a Proposition:
\begin{proposition}\label{prop::radial_restriction_characterization}
For any radial restriction sample $K$, there exist $\alpha,\beta\in\R$ such that
\[\PP\bigl[K\cap A=\emptyset\bigr]=|\Phi_A'(0)|^{\alpha}\Phi_A'(1)^{\beta} \quad \text{ for all }\quad A\in\LA^r.\]
\end{proposition}

Note that Lemma \ref{lem::unicity_law} conversely shows that for any $\alpha$ and $\beta$, there exists at most one law (for $K$) that satisfies this property. When it exists, we call it $\PP (\alpha, \beta)$. An example is provided by radial SLE$_{8/3}$ (see Equation \eqref{eqn::radial_sle_8/3}) that corresponds to $\PP(5/48, 5/8)$.

\medbreak

The first part of the proof of the proposition will be devoted to show that $\theta \mapsto \nu (\theta)$ is a continuously differentiable function. Once this will have been established, it will be possible to use ``commutation relation ideas'' inspired by the formal calculations in \cite{LawlerSchrammWernerConformalRestriction} and by Dub\'edat's paper \cite{DubedatCommutationSLE}.

In order to prove this proposition, it will in fact be a little easier to work in the upper half plane instead of the unit disc. Consider
the conformal map $\varphi_0(z)=i(1-z)/(1+z)$ which maps $\U$ onto $\HH$ and sends $1$ to $0$, $0$ to $i$. A radial restriction sample in $\HH$ (with specified points $0$ and $i$) is just the image of radial restriction sample in $\U$ under the conformal map $\varphi_0$. For $x\in \C,r>0$, We denote $B(x,r)$ as the disc centered at $x$ with radius $r$.

Fix $x\in\R\setminus\{0\},$ let $0<\eps<|x|$. Then
\[g_{x,\eps}(z):=z+\frac{\eps^2}{z-x}\]
is a conformal map from $\HH\setminus B(x,\eps)$ onto $\HH$. Define
\[f_{x,\eps}(z)= b\frac{g_{x,\eps}(z)-c}{b^2+(c-a)(g_{x,\eps}(z)-a)}\]
where $a=\Re(g_{x,\eps}(i)), b=\Im(g_{x,\eps}(i)), c=g_{x,\eps}(0)$. Then $f_{x,\eps}$ is the conformal map from $\HH\setminus B(x,\eps)$ onto $\HH$ that preserves 0 and $i$.

We use the notation $f\lesssim g$ to express that $f/g$ is bounded by universal constant, $f\gtrsim g$ to express $g\lesssim f$, and $f\asymp g$ to express $f\lesssim g$ and $f\gtrsim g$.

\begin{lemma}\label{lem::lambda_definition}
Let $K$ be a radial restriction sample in $\HH$. For any $x\in\R\setminus\{0\}$, the following limit exists
\[\lim_{\eps\to 0}\frac{1}{\eps^2}\PP\bigl[K\cap B(x,\eps)\neq \emptyset\bigr].\]
We denote the limit as $\lambda(x)$, we have further that $\lambda(x)\in (0,\infty)$.
\end{lemma}
\begin{proof}
Fix $x\in(0,\infty)$ and let $\theta\in(0,\pi)$ such that $x=\sin\theta/(1+\cos\theta)$. Let $\eta^x$ be the perfect radial curve in $\HH$ started from $x$ and ended at $i$ which is the image of the perfect radial curve in $\U$ started from $e^{i\theta}$ and ended at the origin under the conformal map $\varphi_0$.
\begin{figure}[ht!]
\begin{center}
\includegraphics[width=0.5\textwidth]{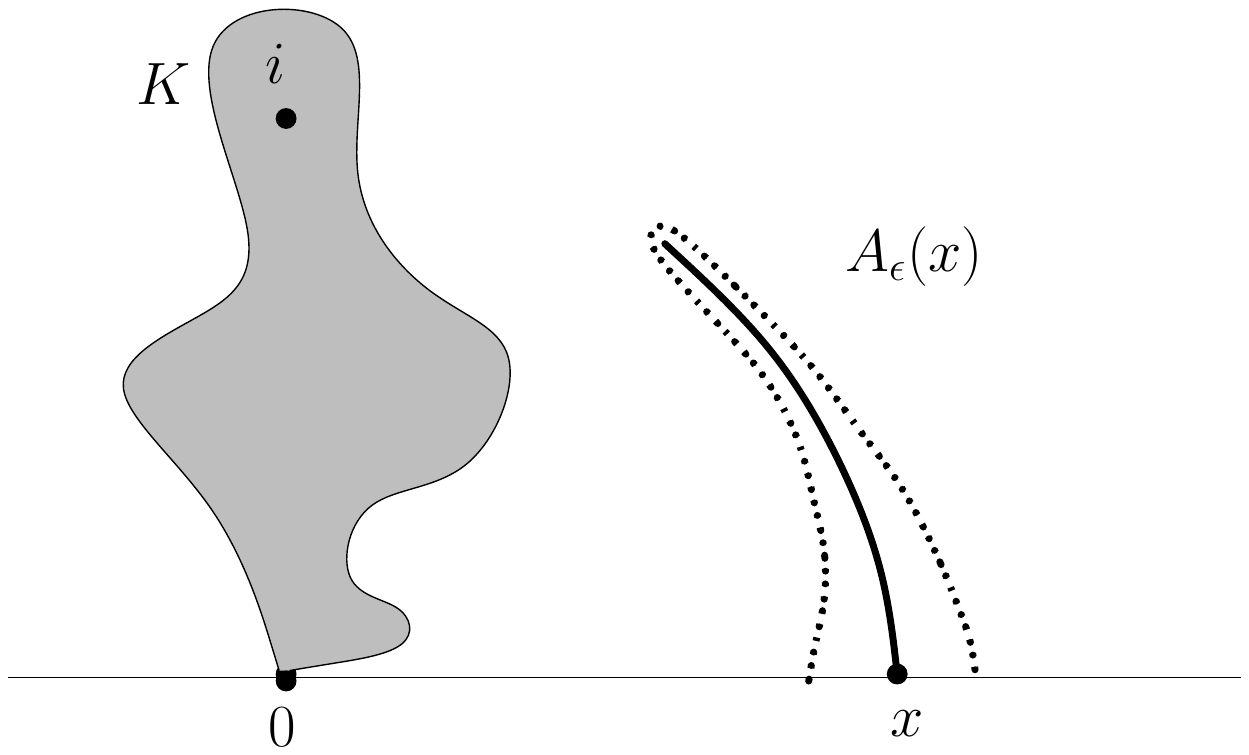}
\end{center}
\caption{\label{fig::approximation_tube} $A_{\eps}(x)$ converges to $\eta^{x}([0,t_x])$ in Hausdorff metric.}
\end{figure}
For $\eps>0$, define $N(\eps)=\lceil\eps^{-2}\rceil$. And $\varphi_1=\cdots=\varphi_N=f_{x,\eps}$. Let $\Phi_{\eps}=\varphi_{N(\eps)}\circ\cdots\circ\varphi_1$. Note that $\Phi_{\eps}$ is a conformal map from $H:=\varphi_1^{-1}\circ\cdots\circ\varphi_{N(\eps)}^{-1}(\HH)$ onto $\HH$ that preserves $i$ and 0. Define $A_{\eps}(x)=\overline{\HH\setminus H}$ (see Figure \ref{fig::approximation_tube}). Then we have that,
\[A_{\eps}(x)\to \eta^x([0,t_x])\quad\text{as}\quad \eps\to 0\]
where $t_x=(1+\cos\theta)^2$ by direct computation of the capacity of $A_{\eps}(x)$ in $\HH$ seen from $i$. And the convergence is under Hausdorff metric. Furthermore, we have that,
\[A_{\eps}(x)\supset \eta^x([0,t_x]).\]
In fact, this is true when $|x|$ is large where $\eta^x$ is very close to vertical line. And this fact does not depend on the location of $x$.

Define $p_{\eps}(x)=\PP\bigl[K\cap B(x,\eps)\neq \emptyset\bigr]$. On the one hand, from conformal restriction property, we know that
\[\PP\bigl[K\cap A_{\eps}(x)=\emptyset\bigr]=(1-p_{\eps}(x))^{N(\eps)}.\]
On the other hand, we know that
\[\PP\bigl[K\cap A_{\eps}(x)=\emptyset\bigr]\to \PP\bigl[K\cap\eta^x([0,t_x])=\emptyset\bigr]=\exp(-\nu(\theta)t_x) \quad\text{as}\quad\eps\to 0.\]
Compare these two relations, we have that
\[\lim_{\eps\to 0}N(\eps)\log(1-p_{\eps}(x))=-\nu(\theta)(1+\cos\theta)^2.\] This completes the proof. And we further know that
\begin{equation}\label{eqn::relation_lambda_nu}
\lambda(\frac{\sin\theta}{1+\cos\theta})=\nu(\theta)(1+\cos\theta)^2.
\end{equation}
\end{proof}

Lemmas \ref{lem::lambda_uniformconvergence} to \ref{lem::lambda_derivative_convergence} show the regularities of the function $\lambda$. To make the proofs easier to follow, we summarize the notations and the basic properties here.
\begin{equation}\label{eqn::lambdarelatednotations}
\begin{split}
p_{\eps}(x)&:=\PP[K\cap B(x,\eps)\neq\emptyset]\\
\lambda(x)&:=-\log\PP[K\cap\eta^x([0,t_x])=\emptyset]\\
\lambda_{\eps}(x)&:=-\log\PP[K\cap A_{\eps}(x)=\emptyset]=-N(\eps)\log(1-p_{\eps}(x))
\end{split}\end{equation}
Let $F^x$ (resp. $F^x_{\eps}$) be the conformal map from $\HH\setminus\eta^x([0,t_x])$ (resp. $\HH\setminus A_{\eps}(x)$) onto $\HH$ that preserves $i$ and 0. Fix a compact interval $I\subset(-\infty,0)\cup(0,\infty)$.

We know that
\[A_{\eps}(x)\supset \eta^x([0,t_x]),\quad \text{and}\quad A_{\eps}(x)\to \eta^x([0,t_x])\quad\text{as}\quad \eps\to 0.\]
Thus $\lambda_{\eps}(x)$ and $p_{\eps}(x)/\eps^2$ converge to $\lambda(x)$ as $\eps$ goes to zero. Since $x\mapsto \eta^x([0,t_x])$ is continuous in Hausdorff metric, we also know that $\lambda$ is a continuous function.
\begin{lemma}\label{lem::lambda_uniformconvergence}
The functions $\lambda_{\eps}(\cdot)$ converges to $\lambda(\cdot)$ uniformly over $I$. Furthermore, for $x\in I$,
\begin{equation}\label{eqn::p_uniformasymptotic}
p_{\eps}(x)\asymp \eps^2
\end{equation}
where the constants in $\asymp$ only depend on $I$.
\end{lemma}
\begin{proof}
For $x\in I, \eps>0$, we have that
\begin{eqnarray}\label{eqn::lambda_difference_analysis}
\lefteqn{\exp(-\lambda(x))-\exp(-\lambda_{\eps}(x))}\notag\\
&=&\PP[K\cap \eta^x([0,t_x])=\emptyset]-\PP[K\cap A_{\eps}(x)=\emptyset]\notag\\
&=&\PP[K\cap \eta^x([0,t_x])=\emptyset, K\cap A_{\eps}(x)\neq\emptyset]\notag\\
&=&\exp(-\lambda(x))\PP[K\cap F^x(A_{\eps}(x))\neq\emptyset].
\end{eqnarray}
Now we will argue that the set $F^x(A_{\eps}(x))$ is uniformly small. The conformal map $F^x$ is Lipschitz when it is bounded away from the tip of $\eta^x([0,t_x])$, whereas it is $1/2$-H\"{o}lder at the tip of $\eta^x([0,t_x])$. However, the semi-disc at the tip of $A_{\eps}(x)$ also has radius of order $\eps^2$, i.e. the radii of the $N(\eps)$ semi-discs in $A_{\eps}(x)$ decrease gradually and the last one has radius bounded by a universal constant times $\eps^2$. (This fact is implicitly used later in the paper.) Thus, there exist compact interval $J$ and constant $C$ depending on $I$ such that $F^x(A_{\eps}(x))$ can be covered by $J^{C\eps}$ which is $C\eps$-neighborhood of $J$. Then
\[|1-\exp(\lambda(x)-\lambda_{\eps}(x))|\le \PP[K\cap J^{C\eps}\neq\emptyset]\]
where $\PP[K\cap J^{C\eps}\neq\emptyset]$ converges to zero as $\eps$ goes to zero. This completes the proof of uniform convergence.

Equation (\ref{eqn::p_uniformasymptotic}) can then be derived by combining the uniform convergence, the relation between $\lambda_{\eps}(x)$ and $p_{\eps}(x)$ in Equation (\ref{eqn::lambdarelatednotations}), and the continuity of $\lambda$.
\end{proof}
\begin{lemma}\label{lem::lambda_holdercontinuity}
For any $x, y\in I$, and $\eps>0,\delta>0$, we have
\begin{equation}\label{eqn::lambda_holdercontinuity_eps}
|\lambda_{\eps}(x)-\lambda_{\delta}(x)|\lesssim |\delta-\eps|\end{equation}
\begin{equation}\label{eqn::lambda_holdercontinuity_x}
|\lambda_{\eps}(x)-\lambda_{\eps}(y)|\lesssim |x-y|\end{equation}
where the constant in $\lesssim$ only depends on $I$. In particular, we have
\[|\lambda(x)-\lambda(y)|\lesssim |x-y|\]
where the constant in $\lesssim$ only depends on $I$. Thus, $\lambda$ is almost everywhere differentiable, i.e. $\lambda$ is differentiable except on a Lebesgue measure zero set.
\end{lemma}
\begin{proof}
We will show Equation (\ref{eqn::lambda_holdercontinuity_eps}) and then Equation (\ref{eqn::lambda_holdercontinuity_x}) can be proved similarly.

Suppose $\delta>\eps>0$. Recall that $F^x_{\eps}$ is the conformal map from $\HH\setminus A_{\eps}(x)$ onto $\HH$ that fixes $i$ and 0. Then we have that
\begin{eqnarray*}
\lefteqn{\exp(-\lambda_{\eps}(x))-\exp(-\lambda_{\delta}(x))}\\
&=&\PP[K\cap A_{\eps}(x)=\emptyset]-\PP[K\cap A_{\delta}(x)=\emptyset]\\
&=&\PP[K\cap A_{\eps}(x)=\emptyset, K\cap A_{\delta}(x)\neq\emptyset]\\
&=&\exp(-\lambda_{\eps}(x))\PP[K\cap F_{\eps}^x(A_{\delta}(x))\neq\emptyset].
\end{eqnarray*}
There exists a constant $C$ depending only on $I$ such that $F_{\eps}^x(A_{\delta}(x))$ can be covered by $\lceil C/|\delta-\eps|\rceil$ balls of radius $C|\delta-\eps|$. Combine with Equation (\ref{eqn::p_uniformasymptotic}), we have that
\[\PP[K\cap F^x_{\eps}(A_{\delta}(x))\neq\emptyset]\lesssim |\delta-\eps|.\]
This completes the proof.
\end{proof}
\begin{lemma}\label{lem::lambda_derivative_convergence}
For any $x,y\in I$ and $\eps>0$, we have
\[|\left(\lambda_{\eps}(x)-\lambda(x)\right)-\left(\lambda_{\eps}(y)-\lambda(y)\right)|\lesssim |x-y|\eps\]
where the constant in $\lesssim$ only depends on $I$.
\end{lemma}
\begin{proof}
In Equation (\ref{eqn::lambda_difference_analysis}), we already see that
\[1-\exp(\lambda(x)-\lambda_{\eps}(x))=\PP[K\cap F^x(A_{\eps}(x))\neq\emptyset],\]
\[1-\exp(\lambda(y)-\lambda_{\eps}(y))=\PP[K\cap F^y(A_{\eps}(y))\neq\emptyset].\]
Thus
\begin{eqnarray*}
\lefteqn{\exp(\lambda(y)-\lambda_{\eps}(y))-\exp(\lambda(x)-\lambda_{\eps}(x))}\\
&=&\PP[K\cap F^x(A_{\eps}(x))\neq\emptyset]-\PP[K\cap F^y(A_{\eps}(y))\neq\emptyset]\\
&=&\PP[K\cap F^x(A_{\eps}(x))\neq\emptyset, K\cap F^y(A_{\eps}(y))=\emptyset]\\
&&-\PP[K\cap F^x(A_{\eps}(x))=\emptyset, K\cap F^y(A_{\eps}(y))\neq\emptyset].
\end{eqnarray*}

There exists constant $C$ depending only on $I$ such that the set $F^y(A_{\eps}(y))\setminus F^x(A_{\eps}(x))$
can be covered by $\lceil C|x-y|/\eps\rceil$ balls of radius $C\eps$. Together with Equation (\ref{eqn::p_uniformasymptotic}), we have that
\[\PP[K\cap F^x(A_{\eps}(x))=\emptyset, K\cap F^y(A_{\eps}(y))\neq\emptyset]\lesssim |x-y|\eps.\]
Thus
\[|\exp(\lambda(y)-\lambda_{\eps}(y))-\exp(\lambda(x)-\lambda_{\eps}(x))|\lesssim |x-y|\eps\]
which completes the proof.
\end{proof}

Fix $x,y\in\R\setminus\{0\}$, define
\[F(x,y)=\lim_{\eps\to 0}\frac{1}{\eps^2}(f_{x,\eps}(y)-y),\quad G(x,y)=\lim_{\eps\to 0}\frac{1}{\eps^2}(f'_{x,\eps}(y)-1).\]
By direct computation, we have that
\begin{equation}\label{eqn::two_kernels}
F(x,y)=\frac{1+x^2+y^2+xy}{x(1+x^2)}+\frac{1}{y-x},\quad G(x,y)=\frac{x+2y}{x(1+x^2)}-\frac{1}{(y-x)^2}.
\end{equation}

\begin{lemma} \label{lem::lambda_differentiability}
The function $\lambda$ defined in Lemma \ref{lem::lambda_definition} is differentiable in $x\in (-\infty,0)\cup(0,\infty)$ and satisfies the following \textit{commutation relation}: for any $x,y\in\R\setminus\{0\}$,
\begin{equation}\label{eqn::commutation_relation}
\lambda'(y)F(x,y)+2\lambda(y)G(x,y)=\lambda'(x)F(y,x)+2\lambda(x)G(y,x).
\end{equation}
\end{lemma}
\begin{proof}
From Lemma \ref{lem::lambda_holdercontinuity}, $\lambda$ is locally Lipschitz continuous in $\R\setminus\{0\}$, it is differentiable almost everywhere, and there exists an integrable function $\omega$ such that, $\lambda'(x)=\omega(x)$ at the point $x$ at which $\lambda$ is differentiable, and, for any $x>y>0$ (or $y<x<0$),
\[\lambda(x)-\lambda(y)=\int_y^x \omega(u)du.\]
\medbreak
Consider two points $x,y$ at which $\lambda$ is differentiable. Let $\eps>0,\delta>0$.
\begin{eqnarray*}\lefteqn{\PP[K\cap B(x,\eps)\neq\emptyset, K\cap B(y,\delta)\neq\emptyset]}\\
&=&\PP[K\cap B(x,\eps)=\emptyset,K\cap B(y,\delta)=\emptyset]-1+p_{\eps}(x)+p_{\delta}(y)\\
&=&\PP\bigl[K\cap B(x,\eps)=\emptyset\bigr]\times\PP\bigl[K\cap f_{x,\eps}(B(y,\delta))=\emptyset\bigr]-1+p_{\eps}(x)+p_{\delta}(y)\\
&=&p_{\delta}(y)-\PP\bigl[K\cap f_{x,\eps}(B(y,\delta))\neq\emptyset\bigr](1-p_{\eps}(x)).\end{eqnarray*}
Divide by $\eps^2\delta^2$ and take the limit, we have that
\begin{eqnarray*}
\lefteqn{\lim_{\eps\to 0}\lim_{\delta\to 0}\frac{1}{\eps^2\delta^2}\PP[K\cap B(x,\eps)\neq\emptyset, K\cap B(y,\delta)\neq\emptyset]}\\
&=&\lim_{\eps\to 0}\lim_{\delta\to 0}\frac{1}{\eps^2\delta^2}\left(p_{\delta}(y)-\PP\bigl[K\cap f_{x,\eps}(B(y,\delta))\neq\emptyset\bigr](1-p_{\eps}(x))\right)\\
&=&\lim_{\eps\to 0}\frac{1}{\eps^2}\left(\lambda(y)-\lambda(f_{x,\eps}(y))|f_{x,\eps}'(y)|^2(1-p_{\eps}(x))\right)\\
&=& \lambda(x)\lambda(y)-\lambda'(y)F(x,y)-2\lambda(y)G(x,y).\end{eqnarray*}
Lemma \ref{lem::lambda_exchangeorder_limits} guarantees that we are allowed to exchange the order of the limits, i.e.
\begin{align*}&\lim_{\eps\to 0}\lim_{\delta\to 0}\frac{1}{\eps^2\delta^2}\PP[K\cap B(x,\eps)\neq\emptyset, K\cap B(y,\delta)\neq\emptyset]\\=&\lim_{\delta\to 0}\lim_{\eps\to 0}\frac{1}{\eps^2\delta^2}\PP[K\cap B(x,\eps)\neq\emptyset, K\cap B(y,\delta)\neq\emptyset].\end{align*}
Then, by the symmetry,
we get Equation \eqref{eqn::commutation_relation} for the points $x,y$ at which $\lambda$ is differentiable.

Fix $y$ in Equation \eqref{eqn::commutation_relation}, we have
\[\lambda'(x)=(\lambda'(y)F(x,y)+2\lambda(y)G(x,y)-2\lambda(x)G(y,x))/F(y,x).\]
The right side is continuous in $x\in \R\setminus\{0,y\}$. Thus we can extend $\omega$ to $\R\setminus\{0,y\}$ by the right side. Then it is clear that $\omega$ is a continuous function in $\R\setminus\{0\}$ and in particular, this implies that $\lambda$ is differentiable everywhere in $\R\setminus\{0\}$ and the derivative satisfies Equation \eqref{eqn::commutation_relation} for any points $x,y\in\R\setminus\{0\}$.

\end{proof}

\begin{lemma}\label{lem::lambda_exchangeorder_limits}
Fix two compact intervals $I,J\subset (-\infty,0)\cup(0,\infty)$. Suppose that $x\in I, y\in J$ and that $\lambda$ is differentiable at $y$, then we have that
\begin{align*}
\PP&[K\cap B(x,\eps)\neq\emptyset,K\cap B(y,\delta)\neq\emptyset]\\
&-\eps^2\delta^2\left(\lambda(x)\lambda(y)-
\lambda'(y)F(x,y)-2\lambda(y)G(x,y)\right)=o(\eps^2\delta^2).\end{align*}
\end{lemma}
\begin{proof}
Set $\tilde{y}=f_{x,\eps}(y)$ and $\tilde{\delta}=f'_{x,\eps}(y)\delta$. Clearly
\begin{equation}\label{eqn::ydelta_approximation}
\tilde{y}=y+\eps^2F(x,y)+o(\eps^2),\quad \tilde{\delta}=\delta(1+\eps^2G(x,y))+o(\eps^2\delta).\end{equation}
Note that
\begin{eqnarray*}
\lefteqn{\PP[K\cap B(x,\eps)\neq\emptyset,K\cap B(y,\delta)\neq\emptyset]}\\
&=&p_{\delta}(y)-\PP[K\cap f_{x,\eps}(B(y,\delta))\neq\emptyset]+p_{\eps}(x)\PP[K\cap f_{x,\eps}(B(y,\delta))\neq\emptyset]
\end{eqnarray*}
The conclusion can be derived by combing the following four relations.
\begin{equation}\label{eqn::lambda_exchangeorder_estimate1}
p_{\delta}(y)-p_{\tilde{\delta}}(y)+2\eps^2\delta^2\lambda(y)G(x,y)=o(\eps^2\delta^2)
\end{equation}
\begin{equation}\label{eqn::lambda_exchangeorder_estimate2}
p_{\tilde{\delta}}(y)-p_{\tilde{\delta}}(\tilde{y})+\eps^2\delta^2\lambda'(y)F(x,y)=o(\eps^2\delta^2)
\end{equation}
\begin{equation}\label{eqn::lambda_exchangeorder_estimate3}
p_{\tilde{\delta}}(\tilde{y})-\PP[K\cap f_{x,\eps}(B(y,\delta))\neq\emptyset]=o(\eps^2\delta^2)
\end{equation}
\begin{equation}\label{eqn::lambda_exchangeorder_estimate4}
p_{\eps}(x)p_{\tilde{\delta}}(\tilde{y})-\eps^2\delta^2\lambda(x)\lambda(y)=o(\eps^2\delta^2)
\end{equation}
We will show Equations (\ref{eqn::lambda_exchangeorder_estimate1}) to (\ref{eqn::lambda_exchangeorder_estimate4}) one by one.

Equation (\ref{eqn::lambda_exchangeorder_estimate1}) is equivalent to the following
\[\lambda_{\delta}(y)-\lambda_{\tilde{\delta}}(y)(1+2\eps^2G(x,y))+2\eps^2\lambda(y)G(x,y)=o(\eps^2).\]
Note that
\begin{eqnarray*}
\lefteqn{\lambda_{\delta}(y)-\lambda_{\tilde{\delta}}(y)(1+2\eps^2G(x,y))+2\eps^2\lambda(y)G(x,y)}\\
&=&\lambda_{\delta}(y)-\lambda_{\tilde{\delta}}(y)+2\eps^2G(x,y)(\lambda(y)-\lambda_{\tilde{\delta}}(y))\\
&=&\lambda_{\delta}(y)-\lambda_{\tilde{\delta}}(y)+o(\eps^2).
\end{eqnarray*}
By Equation (\ref{eqn::lambda_holdercontinuity_eps}), we have that
\[\lambda_{\delta}(y)-\lambda_{\tilde{\delta}}(y)=O(|\delta-\tilde{\delta}|)=O(\eps^2\delta)=o(\eps^2).\]
This completes the proof of Equation (\ref{eqn::lambda_exchangeorder_estimate1}).

Equation (\ref{eqn::lambda_exchangeorder_estimate2}) is equivalent to the following
\[(\lambda_{\tilde{\delta}}(y)-\lambda_{\tilde{\delta}}(\tilde{y}))(1+2\eps^2G(x,y))+\eps^2\lambda'(y)F(x,y)=o(\eps^2).\]
Note that
\begin{eqnarray*}
\lefteqn{(\lambda_{\tilde{\delta}}(y)-\lambda_{\tilde{\delta}}(\tilde{y}))(1+2\eps^2G(x,y))+\eps^2\lambda'(y)F(x,y)}\\
&=&\lambda_{\tilde{\delta}}(y)-\lambda_{\tilde{\delta}}(\tilde{y})+\eps^2\lambda'(y)F(x,y)+o(\eps^2)\\
&=&\left(\lambda_{\tilde{\delta}}(y)-\lambda(y)-\lambda_{\tilde{\delta}}(\tilde{y})+\lambda(\tilde{y})\right)+\lambda(y)-\lambda(\tilde{y})
+\eps^2\lambda'(y)F(x,y)+o(\eps^2)\\
&=&\left(\lambda_{\tilde{\delta}}(y)-\lambda(y)-\lambda_{\tilde{\delta}}(\tilde{y})+\lambda(\tilde{y})\right)+\left(\lambda(y)-\lambda(\tilde{y})
+\lambda'(y)(\tilde{y}-y)\right)+o(\eps^2)
\end{eqnarray*}
By Lemma \ref{lem::lambda_derivative_convergence}, we have that
\[\lambda_{\tilde{\delta}}(y)-\lambda(y)-\lambda_{\tilde{\delta}}(\tilde{y})+\lambda(\tilde{y})=O(|y-\tilde{y}|\tilde{\delta})=o(\eps^2).\]
Since $\lambda$ is differentiable at $y$, we have that
\[\lambda(y)-\lambda(\tilde{y})
+\lambda'(y)(\tilde{y}-y)=o(\eps^2).\]
These complete the proof of Equation (\ref{eqn::lambda_exchangeorder_estimate2})

For Equation (\ref{eqn::lambda_exchangeorder_estimate3}), we have that
\begin{eqnarray*}
\lefteqn{p_{\tilde{\delta}}(\tilde{y})-\PP[K\cap f_{x,\eps}(B(y,\delta))\neq\emptyset]}\\
&=&\PP[K\cap B(\tilde{y},\tilde{\delta})\neq\emptyset]-\PP[K\cap f_{x,\eps}(B(y,\delta))\neq\emptyset]\\
&=&\PP[K\cap B(\tilde{y},\tilde{\delta})\neq\emptyset, K\cap f_{x,\eps}(B(y,\delta))=\emptyset]\\
&&-\PP[K\cap B(\tilde{y},\tilde{\delta})=\emptyset, K\cap f_{x,\eps}(B(y,\delta))\neq\emptyset]
\end{eqnarray*}
Note that
\[\PP[K\cap B(\tilde{y},\tilde{\delta})=\emptyset, K\cap f_{x,\eps}(B(y,\delta))\neq\emptyset]\le \PP[K\cap f_{\tilde{y},\tilde{\delta}}(f_{x,\eps}(B(y,\delta)))\neq\emptyset]\]
Set $z=\delta e^{i\theta}$ for $\theta\in [0,\pi]$. Since $f''_{x,\eps}(y)=O(\eps^2)$, we have that
\[f_{x,\eps}(y+z)=f_{x,\eps}(y)+f'_{x,\eps}(y)z+o(\eps^2\delta).\]
Set $\Delta=f_{x,\eps}(y+\delta)-\tilde{y}-\tilde{\delta}$. In fact, $\Delta=o(\eps^2\delta)$. There exists constant $C$ depending only on $I, J$ such that the set $f_{\tilde{y},\tilde{\delta}}(f_{x,\eps}(B(y,\delta)))$ can be covered by $\lceil C\delta/\Delta\rceil$ balls of radius $C\Delta$. Together with Equation (\ref{eqn::p_uniformasymptotic}), we have that
\[\PP[K\cap f_{\tilde{y},\tilde{\delta}}(f_{x,\eps}(B(y,\delta)))\neq\emptyset]\lesssim \delta\Delta=o(\eps^2\delta^2)\]
which completes the proof of Equation (\ref{eqn::lambda_exchangeorder_estimate3}).

Equation (\ref{eqn::lambda_exchangeorder_estimate4}) is equivalent to the following
\[\lambda_{\eps}(x)\lambda_{\tilde{\delta}}(\tilde{y})-\lambda(x)\lambda(y)=o(1)\]
which is clearly true.
\end{proof}

\begin{lemma}\label{lem::lambda_expression}
There exist two constants $c_0,c_2 \ge 0$ such that
\[\lambda(x)=\frac{c_0+c_2 x^2}{x^2(1+x^2)^2}\quad \text{for}\quad x\in\R\setminus\{0\}.\]
\end{lemma}

\begin{proof}
From \eqref{eqn::commutation_relation} and \eqref{eqn::two_kernels}, we know that $\lambda$ is smooth in $(-\infty,0)\cup(0,+\infty)$. In \eqref{eqn::commutation_relation}, fix $x\in\R\setminus\{0\}$, and let $y\to x$. Compare the coefficients of the two sides of the equation, we have that
\begin{align}\label{eqn::lambda_ode}
x^2(1+x^2)^2\lambda'''(x)&+6x(1+x^2)(1+3x^2)\lambda''(x)\notag\\
&+6(1+12x^2+15x^4)\lambda'(x)+24x(2+5x^2)\lambda(x)=0.
\end{align}
%%%%%%%%%%%%%%%%%%%%%%%%%%%%%%%%%%%%
%%keep: details for the computation
%%Set $\Delta=y-x$, then
%%\begin{align*}
%%LHS=&\frac{1}{\Delta^2}(-2\lambda(x))+\frac{1}{\Delta}(-\lambda'(x))+(\frac{1+3x^2}{x(1+x^2)}\lambda'(x)+\frac{6}{1+x^2}\lambda(x))\\
%%&+\Delta (\frac{1}{6}\lambda'''(x)+\frac{1+3x^2}{x(1+x^2)}\lambda''(x)+\frac{9}{1+x^2}\lambda'(x)+\frac{4}{x(1+x^2)}\lambda(x))
%%\end{align*}
%%\begin{align*}
%%RHS=&\frac{1}{\Delta^2}(-2\lambda(x))+\frac{1}{\Delta}(-\lambda'(x))+(\frac{1+3x^2}{x(1+x^2)}\lambda'(x)+\frac{6}{1+x^2}\lambda(x))\\
%%&+\Delta (\frac{-(1+3x^2+6x^4)}{x^2(1+x^2)^2}\lambda'(x)+\frac{-4(1+4x^2)}{x(1+x^2)^2}\lambda(x))
%%\end{align*}
%%%%%%%%%%%%%%%%%%%%%%%%%%%%%%%%%%%%%%
Set $P(x)=x^2(1+x^2)^2\lambda(x)$, then \eqref{eqn::lambda_ode} is equivalent to
\[P(x)'''=0.\]
Together with the symmetry in $\lambda$, we know that, there exist constants $c_0,c_1,c_2$ such that
\[\lambda(x)=\frac{c_0+c_1 x+c_2 x^2}{x^2(1+x^2)^2}\text{ for }x>0; \quad \lambda(x)=\frac{c_0-c_1 x+c_2 x^2}{x^2(1+x^2)^2}\text{ for }x<0.\]
Take $x>0>y$, by \eqref{eqn::commutation_relation}, we have that $c_1=0$. Since $\lambda(x)>0$ for all $x\in\R\setminus\{0\}$, we know that $c_0\ge 0, c_2\ge 0$.
%%%%%%%%%%%%%%%%%%%%%%%%%%%%%%%%%%%%%%
%%keep: details for the computation
%%\begin{align*}
%%\lambda'(x)=\frac{-1}{x^3(1+x^2)^3}(4c_2x^4+5c_1x^3+6c_0x^2+c_1x+2c_0)
%%\end{align*}
%%%%%%%%%%%%%%%%%%%%%%%%%%%%%%%%%%%%%%%%%
\end{proof}

\begin{proof}[Proof of Proposition \ref{prop::radial_restriction_characterization}] Consider a radial restriction sample $K$ in $\U$. Fix $\theta\in(0,\pi)$, let $\nu(\theta)$ be defined through Lemma \ref{lem::definition_nu}. And let $\lambda$ be defined through Lemma \ref{lem::lambda_definition}.
From Lemma \ref{lem::lambda_expression} and Equation \eqref{eqn::relation_lambda_nu}, we have that
\[\nu(\theta)=-\alpha+\frac{\beta}{1-\cos\theta}\]
where $\alpha=(c_0-c_2)/4,\beta=c_0/2$. Recall Equation \eqref{eqn::perfect_two_derivatives}, we have that
\[\PP\bigl[K\cap\eta^{\theta}([0,t])=\emptyset\bigr]=|f'_t(0)|^{\alpha}f'_t(1)^{\beta}.\]
Then the conclusion can be derived by similar explanation as in \cite[Proposition 3.3]{LawlerSchrammWernerConformalRestriction}.
\end{proof}

\section {Admissible range of $(\alpha, \beta)$}

\subsection{Description of $\PP(\alpha, \beta)$'s when $\beta\ge 5/8$}

In order to complete the proof of our main theorem, it now remains to show for which values of $\alpha$ and $\beta$ the previous measure exists.
Note now that from the properties of Poisson point process of Brownian loops, we can deduce the following fact:
\begin{lemma}\label{lem::brownian_loop_soup}
If the radial restriction measure $\PP(\alpha_0,\beta_0)$ exists for some $\alpha_0,\beta_0\in\R$, then for any $\alpha<\alpha_0$, $\PP(\alpha,\beta_0)$ exists, and
furthermore, almost surely for $\PP ( \alpha, \beta_0)$, the origin is not on the boundary of $K$.
\end{lemma}
\begin{proof}
Let $K_0$ be a closed set sampled according to $\PP(\alpha_0,\beta_0)$, and let $(\gamma_j,j\in J)$ be an independent Poisson Point Process with intensity $(\alpha_0-\alpha)\mu_{\U}^0$. We view each loop $\gamma_j$ as the loop with the domain that it surrounds. Then let $K$ be the closure of the union of $K_0$ and all loops in ($\gamma_j,j\in J$). We have that, for any $A\in\LA^r$,
\begin{eqnarray*}\lefteqn{\PP\bigl[K\cap A=\emptyset\bigr]}\\
&=&\PP\bigl[K_0\cap A=\emptyset\bigr]\times\PP\bigl[\gamma_j\cap A=\emptyset,\forall j\in J\bigr]\\
&=&|\Phi_A'(0)|^{\alpha_0}\Phi_A'(1)^{\beta_0}|\Phi_A'(0)|^{\alpha-\alpha_0}=|\Phi_A'(0)|^{\alpha}\Phi_A'(1)^{\beta_0}.\end{eqnarray*}
It is clear that $K$ has the law of $\PP(\alpha,\beta_0)$ and the $0 \notin \partial K$.
\end{proof}

Hence, we have the following result:
\begin{corollary}\label{cor::origin_boundary}
Suppose that a radial restriction measure $\PP(\alpha_0,\beta_0)$ exists for some $\alpha_0,\beta_0\in\R$, and that for this measure, $0 \in \partial K$ almost surely.
Then, $\PP(\alpha,\beta_0)$ does exist if and only if $\alpha \le \alpha_0$.
\end{corollary}

\begin{proof}
Suppose that $\PP(\alpha, \beta_0)$ exists for some $\alpha >  \alpha_0$, and let $K$ be a random set whose law is $\PP(\alpha_0,\beta_0)$. Lemma \ref{lem::brownian_loop_soup} implies that almost surely, $0 \notin \partial K$, which is a contradiction.
On the other hand, Lemma \ref{lem::brownian_loop_soup} shows that $\PP ( \alpha, \beta_0)$ exists for all $\alpha < \alpha_0$.
\end{proof}

In Equation \eqref{eqn::radial_sle_8/3}, we already know the existence of $\PP(\xi(\beta),\beta)$ for $\beta=5/8$. We will construct $\PP(\xi(\beta),\beta)$ for $\beta>5/8$ in Proposition \ref{prop::radial_sample_construction}. Fix $\rho>0$. Let $(g_t, t\ge 0)$ be the radial Loewner chain SLE$_{8/3}(\rho)$ generated by the driving function $(W_t,t\ge 0)$, and $\eta$ be the corresponding radial curve. Recall that $W$ is the solution to the system of SDEs \eqref{eqn::radial_loewner_sde}. To simplify notation, we denote $\theta_t=(W_t-V_t)/2$.
For any $A\in\LA^r$, let $\tau_A$ be the first time that $\eta$ hits $A$. For any $t<\tau_A$, let $h_t$ be the conformal map from $\U\setminus g_t(A)$ onto $\U$ such that $h_t(0)=0,h_t(e^{iW_t})=e^{iW_t}$. Then
\begin{lemma}
\begin{equation}\label{eqn::mart_nonintersection}
M_t:=|h'_t(0)|^{\alpha}\times|h'_t(e^{iW_t})|^{\frac{5}{8}}\times|h'_t(e^{iV_t})|^\gamma\times Z_t^{\frac{3}{8}\rho}
\end{equation}
is a local martingale where
\[Z_t=\frac{\sin\vartheta_t}{\sin\theta_t},\quad \vartheta_t=\frac{1}{2}\arg(h_t(e^{iW_t})/h_t(e^{iV_t})),\]
\[\alpha=\frac{5}{48}+\frac{3}{64}\rho(\rho+4),\quad \gamma=\frac{1}{32}\rho(3\rho+4),\quad \beta=\frac{5}{8}+\gamma+\frac{3}{8}\rho=\frac{1}{32}(\rho+2)(3\rho+10).\] Note that $\alpha=\xi(\beta)$.
\end{lemma}
\begin{proof}
Define $\phi_t(z)=-i\log h_t(e^{iz})$ where $\log$ denotes the branch of the logarithm such that $-i\log h_t(e^{iW_t})=W_t$. Then
\[|h'_t(e^{iW_t})|=\phi'_t(W_t),\quad |h'_t(e^{iV_t})|=\phi'_t(V_t),\quad \vartheta_t=(\phi_t(W_t)-\phi_t(V_t))/2.\]
To simplify the notations, we set $X_1=\phi'_t(W_t), X_2=\phi''_t(W_t), Y_1=\phi'_t(V_t)$. By It\^o formula, we have that
\begin{equation*}
\begin{split}
d\phi_t(W_t)&=\sqrt{8/3}X_1dB_t+\left(-\frac{5}{3}X_2+\frac{\rho}{2}X_1\cot\theta_t\right)dt,\\
d\phi_t(V_t)&=-X_1^2\cot\vartheta_tdt,\\
d\phi'_t(W_t)&=\sqrt{8/3}X_2dB_t+\left(\frac{\rho}{2}X_2\cot\theta_t+\frac{X_2^2}{2X_1}+\frac{X_1-X_1^3}{6}\right)dt, \\
d\phi'_t(V_t)&=\left(-\frac{1}{2}X_1^2Y_1\frac{1}{\sin^2\vartheta_t}+\frac{1}{2}Y_1\frac{1}{\sin^2\theta_t}\right)dt,\\
d\theta_t&=\frac{\sqrt{8/3}}{2}dB_t + \frac{\rho+2}{4}\cot\theta_t dt,\\
%%$$d\sin\theta_t=\frac{\sqrt{8/3}}{2}\cos\theta_t dB_t+\left(\frac{\rho+2}{4}\frac{1}{\sin\theta_t}-\frac{3\rho+10}{12}\sin\theta_t dt\right)$$
d\vartheta_t&=\frac{\sqrt{8/3}}{2}X_1dB_t+\left(-\frac{5}{6}X_2+\frac{1}{2}X_1^2\cot\vartheta_t+\frac{\rho}{4}X_1\cot\theta_t\right)dt.
\end{split}\end{equation*}
%%\begin{align*}d\sin\vartheta_t=&\frac{\sqrt{8/3}}{2}X_1\cos\vartheta_tdB_t\\
%%&+\left(-\frac{5}{6}X_2\cos\vartheta_t+X_1^2(\frac{1}{2\sin\vartheta_t}-\frac{5}{6}\sin\vartheta_t)+\frac{\rho}{4}X_1\cot\theta_t\cos\vartheta_t\right)dt\end{align*}
%%\begin{align*}
%%&\frac{dZ_t}{Z_t}=\frac{\sqrt{8/3}}{2}\left(X_1\cot\vartheta_t-\cot\theta_t\right)dB_t\\
%%&+\left(-\frac{5}{6}X_2\cot\vartheta_t+X_1^2(\frac{1}{2\sin^2\vartheta_t}-\frac{5}{6})+\frac{3\rho-8}{12}X_1\cot\theta_t\cot\vartheta_t+\frac{2-3\rho}{12}\cot^2\theta_t+\frac{1}{3}\right)dt
%%\end{align*}
And note that
\[|h_t'(0)|^{\alpha}=|\Phi_A'(0)|^{\alpha}\exp\left(\alpha(\int_0^t ds|h'_s(e^{iW_s})|^2-t)\right).\]
So that
\[dM_t=\frac{\sqrt{8/3}}{16}M_t\left(10\frac{X_2}{X_1}+3\rho(X_1\cot\vartheta_t-\cot\theta_t)\right)dB_t.\]
\end{proof}

\begin{figure}[ht!]
\begin{center}
\includegraphics[width=0.4\textwidth]{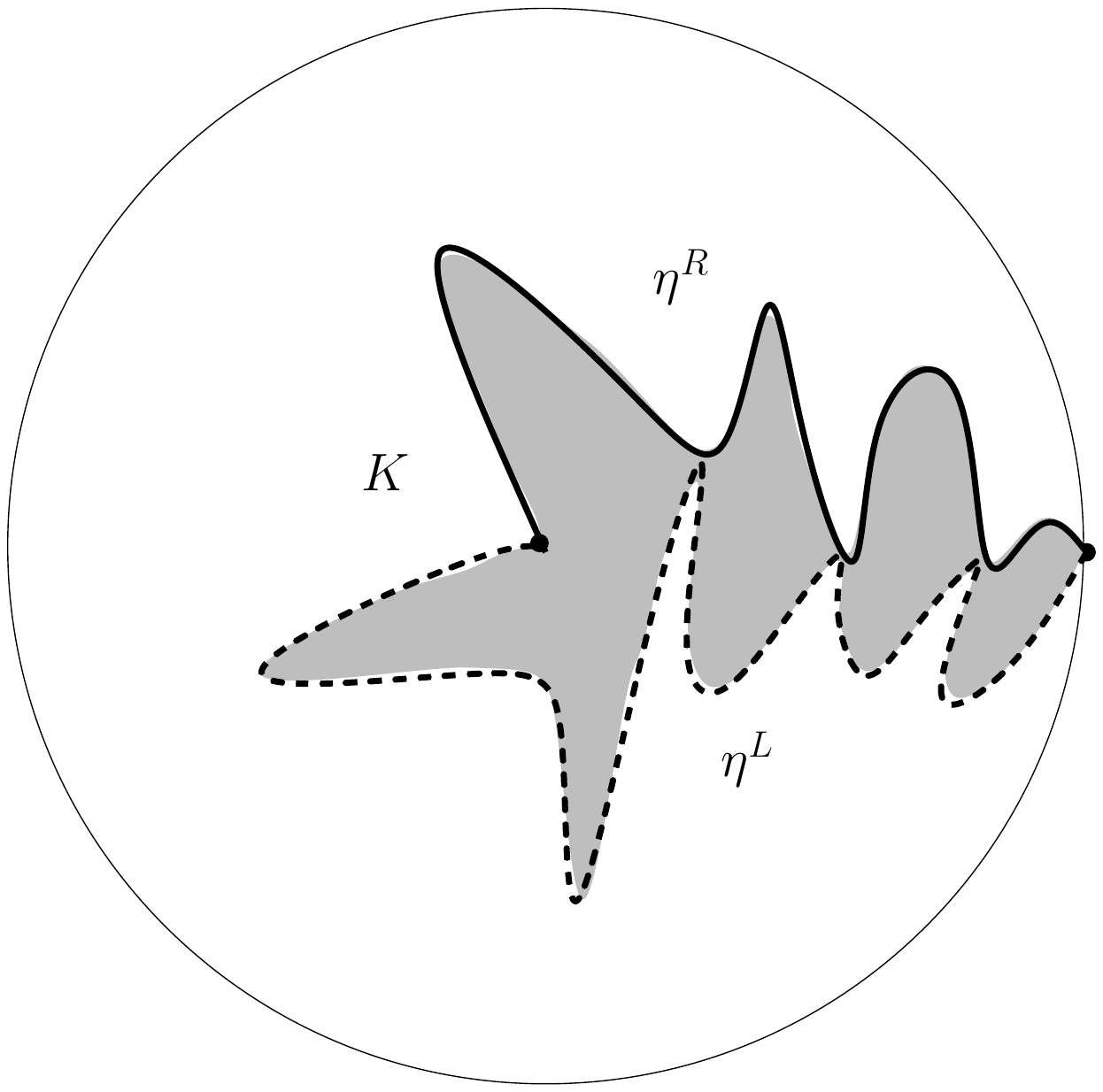}
\end{center}
\caption{\label{fig::construction_maximal_sample}$\eta^R$ is a radial SLE$_{8/3}(\rho)$ in $\U$ from 1 to 0. Conditioned on $\eta^R$, $\eta^L$ is a chordal SLE$_{8/3}^R(\rho-2)$ in $\U\setminus\eta^R([0,\infty])$ from 1 to 0. $K$ is the closure of the union of domains between the two curves.}
\end{figure}

\begin{proposition}\label{prop::radial_sample_construction}
For $\beta>5/8$, let $\rho=\frac{2}{3}(\sqrt{24\beta+1}-1)-2>0$. Let $\eta^R$ be a radial SLE$_{8/3}(\rho)$ in $\overline{\U}$ from 1 to 0 with force point $1^-$. Given $\eta^R$, let $\eta^L$ be an independent chordal SLE$_{8/3}^R(\rho-2)$ in $\U\setminus\eta^R([0,\infty])$ from $1^-$ to 0. Define $K$ as the closure of the union of the domains between $\eta^R$ and $\eta^L$. Then the law of $K$ is $\PP(\xi(\beta),\beta)$ (that therefore exists) and under this probability measure, $0\in\partial K$ almost surely.
\end{proposition}

Hence, this proves that for $\beta \ge 5/8$, $\PP(\alpha, \beta)$ exists if and only if $\alpha \le \xi (\beta)$.

\begin{proof} (See Figure \ref{fig::construction_maximal_sample})
Let $(g_t,t\ge 0)$ be the radial Loewner chain for $\eta^R$. For any $A\in\LA$, let $\tau_A$ be the first time that $\eta^R$ hits $A$. For any $t<\tau_A$, define $h_t$ as the conformal map from $\U\setminus g_t(A)$ onto $\U$ such that $h_t(0)=0, h_t(e^{iW_t})=e^{iW_t}$. Define the local martingale $M$ as in Equation \eqref{eqn::mart_nonintersection}. When $\rho>0$, $M_t$ is positive and bounded by 1. Thus it is a real martingale. Note that \[M_0=|\Phi'_A(0)|^{\xi(\beta)}\Phi'_A(1)^{\beta}.\]

If $\tau_A<\infty$, then there exists a sequence $t_n\to\tau_A$, such that $\lim_n M_{t_n}=0$.

If $\tau_A=\infty$, then there exists a sequence $t_n\to\infty$, such that (see \cite[Section 5.2]{WernerGirsanov})
\[|h'_{t_n}(0)|\to 1,\quad |h'_{t_n}(e^{iW_{t_n}})|\to 1, \quad Z_{t_n}\to 1,\quad |h'_{t_n}(e^{iV_{t_n}})|^{\gamma}\to \PP\bigl[K\cap A=\emptyset\giv\eta^R\bigr].\]
Thus, almost surely,
\[\lim_{t\to\tau_A}M_t=\PP\bigl[K\cap A=\emptyset\giv\eta^R\bigr]1_{\tau_A=\infty}.\]
As a result
\[\PP\bigl[K\cap A=\emptyset\bigr]=\E(M_{\tau_A})=M_0.\]
\end{proof}

\subsection{Why can $\beta$ not be smaller than $5/8$? }
It remains to show that if $\PP(\alpha,\beta)$ exists, then $\beta\ge 5/8$. In the following we assume that $\PP(\alpha,\beta)$ exists. We are going to show how to use this radial measure to construct a chordal restriction measure of exponent $\beta$, which will then imply that $\beta$ cannot be smaller than $5/8$.

Let $X$ be the collection of compact subsets $K$ of $\overline{\U}$ such that $K$ is connected and $\C\setminus K$ is connected. Let $\LA$ be the collection of compact subset $A$ of $\overline{\U}$ such that $A=\overline{\U\cap A}$, $\U\setminus A$ is simply connected. Endow $X$ with the $\sigma$-field generated by the events $\LC(A):=(K\in X: K\cap A=\emptyset)$ for $A\in\LA$. This $\sigma$-field coincides with the $\sigma$-field generated by Hausdorff metric on $X$. In particular, $X$ is compact since $\overline{\U}$ is compact.

Let $K$ be a radial restriction sample of law $\PP(\alpha,\beta)$. For any $\eps>0$, define the probability measure $\mu_\eps$ on $X$ by
\[\mu_\eps(\LC(A))=\PP\bigl[f_\eps(K)\cap A=\emptyset\bigr]\]
where $A\in\LA$ such that $+1\not\in A,-1\not\in A$ and $f_\eps$ is the conformal map from $\U$ onto itself such that $f_\eps(0)=-1+\eps, f_\eps(1)=1$.

Since $X$ is compact, the sequence $(\mu_\eps,\eps>0)$ is tight, thus there exists a subsequence $(\mu_{\eps_k},k\in\N)$ such that $\eps_k\to 0$ and $\mu_{\eps_k}$ converges weakly to some probability measure $\mu$ on $X$. There two observations:
\begin{itemize}
\item For any $A\in\LA$ such that $+1\not\in A,-1\not\in A$,
\begin{align}\label{eqn::convergence}
\mu_{\eps}(\LC(A))=|\Phi'_\eps(-1+\eps)|^{\alpha}\Phi'_\eps(1)^{\beta}\to \Psi'_A(1)^{\beta}\quad \text{as}\quad \eps\to 0
\end{align}
where $\Phi_\eps$ is the conformal map from $\U\setminus A$ onto $\U$ that preserves $-1+\eps$ and $+1$, $\Psi_A$ is the conformal map from $\U\setminus A$ onto $\U$ that preserves $\pm 1$ and $\Psi'_A(-1)=1$.

\item For any $A\in\LA$ such that $+1\not\in A,-1\not\in A$ and $\delta>0$, define $A_o^{\delta}$ as the open $\delta$-neighborhood of $A$ and $A_i^{\delta}=\overline{\U}\setminus (\U\setminus A)_o^{\delta}$. Note that $A_o^{\delta}$ is open, $A_i^{\delta}$ is closed, $\LC(A_o^{\delta})$ is closed and $\LC(A_i^{\delta})$ is open. Thus
    \begin{align*}
    \mu(\LC(A_i^{\delta})\setminus\LC(A_o^{\delta}))\le \lim_k\mu_{\eps_k}(\LC(A_i^{\delta})\setminus\LC(A_o^{\delta})).
    \end{align*}
    From Equation \eqref{eqn::convergence}, we know that there exists $g(\delta)$ goes to zero as $\delta$ goes to zero and is independent of $\eps$ such that
    \[\mu_{\eps_k}(\LC(A_i^{\delta})\setminus\LC(A_o^{\delta}))=\mu_{\eps_k}(\LC(A_i^{\delta}))-\mu_{\eps_k}(\LC(A_o^{\delta}))\le g(\delta).\] Thus we have that
    \begin{align}\label{eqn::continuity}\mu(\LC(A_i^{\delta})\setminus\LC(A_o^{\delta}))\le g(\delta).\end{align}
\end{itemize}

From Equations \eqref{eqn::convergence} and \eqref{eqn::continuity}, we have that
\[\mu(\LC(A))=\Psi'_A(1)^{\beta}\]
for any $A\in\LA$ such that $\pm 1\not\in A$ and $\Psi_A$ is the conformal map from $\U \setminus A$ onto $\U$ that preserves $\pm 1$ and $\Psi'_A(-1)=1$. Thus $\mu$ is the chordal restriction measure of exponent $\beta$, thus $\beta\ge 5/8$.

This concludes the proof of our main theorem.

\subsection{Concluding remarks}

We would just like to note that all the enumerated results on chordal restriction samples that we have briefly recalled in Section \ref{sec::pre_chordal_restriction} do have a radial restriction counterpart:
The dimension of cut-points is the same (and given by $\beta$ only), the boundaries of radial restriction sample $\PP(\xi(\beta),\beta)$ are radial SLE$_{8/3}(\rho)$ processes, the full-plane Brownian intersection exponents describe the law of radial restriction samples conditioned not to intersect etc. We leave the precise statements and detailed proofs to the interested reader.

\section{Acknowledgements}
H.W.'s work is funded by the Fondation CFM JP Aguilar pour la recherche. The author acknowledges also the support and hospitality of FIM at ETH Z\"urich. The author deeply appreciates the reviewer for his/her constructive comments and suggestions on the previous versions of the paper. 



%% References


%\bibliographystyle{elsarticle-num}
%\bibliography{hao_wu_thesis}

\begin{thebibliography}{00}

\bibitem[Dub07]{DubedatCommutationSLE}
Julien Dub{\'e}dat.
\newblock Commutation relations for {S}chramm-{L}oewner evolutions.
\newblock {\em Comm. Pure Appl. Math.}, 60(12):1792--1847, 2007.

\bibitem[Law05]{LawlerConformallyInvariantProcesses}
Gregory~F. Lawler.
\newblock {\em Conformally invariant processes in the plane}, volume 114 of
  {\em Mathematical Surveys and Monographs}.
\newblock American Mathematical Society, Providence, RI, 2005.

\bibitem[LSW01a]{LawlerSchrammWernerExponent1}
Gregory~F. Lawler, Oded Schramm, and Wendelin Werner.
\newblock Values of {B}rownian intersection exponents. {I}. {H}alf-plane
  exponents.
\newblock {\em Acta Math.}, 187(2):237--273, 2001.

\bibitem[LSW01b]{LawlerSchrammWernerExponent2}
Gregory~F. Lawler, Oded Schramm, and Wendelin Werner.
\newblock Values of {B}rownian intersection exponents. {II}. {P}lane exponents.
\newblock {\em Acta Math.}, 187(2):275--308, 2001.

\bibitem[LSW02]{LawlerSchrammWernerExponent3}
Gregory~F. Lawler, Oded Schramm, and Wendelin Werner.
\newblock Values of {B}rownian intersection exponents. {III}. {T}wo-sided
  exponents.
\newblock {\em Ann. Inst. H. Poincar\'e Probab. Statist.}, 38(1):109--123,
  2002.

\bibitem[LSW03]{LawlerSchrammWernerConformalRestriction}
Gregory Lawler, Oded Schramm, and Wendelin Werner.
\newblock Conformal restriction: the chordal case.
\newblock {\em J. Amer. Math. Soc.}, 16(4):917--955 (electronic), 2003.

\bibitem[LW00]{LawlerWernerUniversalityExponent}
Gregory~F. Lawler and Wendelin Werner.
\newblock Universality for conformally invariant intersection exponents.
\newblock {\em J. Eur. Math. Soc. (JEMS)}, 2(4):291--328, 2000.

\bibitem[LW04]{LawlerWernerBrownianLoopsoup}
Gregory~F. Lawler and Wendelin Werner.
\newblock The {B}rownian loop soup.
\newblock {\em Probab. Theory Related Fields}, 128(4):565--588, 2004.

\bibitem[MS12]{MillerSheffieldIG1}
Jason Miller and Scott Sheffield.
\newblock Imaginary geometry {I}: Interacting {SLE}s.
\newblock 2012.

\bibitem[MS13]{MillerSheffieldIG4}
Jason Miller and Scott Sheffield.
\newblock Imaginary geometry {IV}: interior rays, whole-plane reversibility, and
  space-filling trees.
\newblock 2013.

\bibitem[MW13]{MillerWuSLEIntersection}
Jason Miller and Hao Wu.
\newblock Intersections of {SLE} paths: the double and cut point dimension of
  {SLE}, 2013.

\bibitem[RS05]{RohdeSchrammSLEBasicProperty}
Steffen Rohde and Oded Schramm.
\newblock Basic properties of {SLE}.
\newblock {\em Ann. of Math. (2)}, 161(2):883--924, 2005.

\bibitem[Sch00]{SchrammFirstSLE}
Oded Schramm.
\newblock Scaling limits of loop-erased random walks and uniform spanning
  trees.
\newblock {\em Israel J. Math.}, 118:221--288, 2000.

\bibitem[SW05]{SchrammWilsonSLECoordinatechanges}
Oded Schramm and David~B. Wilson.
\newblock S{LE} coordinate changes.
\newblock {\em New York J. Math.}, 11:659--669 (electronic), 2005.

\bibitem[Wer04]{WernerGirsanov}
Wendelin Werner.
\newblock Girsanov's transformation for {${\rm SLE}(\kappa,\rho)$} processes,
  intersection exponents and hiding exponents.
\newblock {\em Ann. Fac. Sci. Toulouse Math. (6)}, 13(1):121--147, 2004.

\bibitem[Wer08]{WernerSelfavoidingLoop}
Wendelin Werner.
\newblock The conformally invariant measure on self-avoiding loops.
\newblock {\em J. Amer. Math. Soc.}, 21(1):137--169, 2008.

\bibitem[Zha09]{ZhanReversalRadialSLE}
Dapeng Zhan.
\newblock On the reversal of radial {SLE}, {I}: Commutation relations in
  annuli, 2009.

\end{thebibliography}






% \end{thebibliography}

\end{document}